\documentclass{amsart}

\usepackage[T1]{fontenc}
\usepackage[usenames,dvipsnames]{xcolor}
\usepackage{graphicx}
\usepackage{ifthen}
\usepackage{tikz}
\usepackage{subcaption}
\usepackage{lineno}

\usepackage{textcomp}
\usepackage{amsfonts}
\usepackage{amssymb}
\usepackage{amsthm}
\usepackage{stmaryrd}

\newcommand{\naturals}{\mathbb{N}}

\newcommand{\integers}{\mathbb{Z}}

\newtheorem*{corollary*}{Corollary}
\newtheorem*{theorem*}{Theorem}
\newtheorem{theorem}{Theorem}[section]
\newtheorem{lemma}[theorem]{Lemma}

\newtheorem{corollary}[theorem]{Corollary}
\newtheorem{observation}[theorem]{Observation}

\theoremstyle{definition}
\newtheorem{definition}[theorem]{Definition}

\newtheorem*{example*}{Example}
\newtheorem*{convention*}{Convention}

\theoremstyle{remark}

\newboolean{showcomments}
\setboolean{showcomments}{true}

\newcommand{\dl}{\text{DL}}
\newcommand{\Z}{\mathbb{Z}}
\newcommand{\kjcomment}[1]%
    {\ifthenelse{\boolean{showcomments}}%
        {\textcolor{blue}{(N.B:[KJ] #1)}}{}}
\newcommand{\gkcomment}[1]%
    {\ifthenelse{\boolean{showcomments}}%
        {\textcolor{ForestGreen}{(N.B:[GK] #1)}}{}}
\newcommand{\kjcommentm}[1]%
    {\ifthenelse{\boolean{showcomments}}%
        {\marginpar{\textcolor{blue}{(N.B:[KJ] #1)}}}{}}
\newcommand{\gkcommentm}[1]%
    {\ifthenelse{\boolean{showcomments}}%
        {\marginpar{\textcolor{ForestGreen}{(N.B:[GK] #1)}}}{}}
\newcommand{\comment}[1]{\ifthenelse{\boolean{showcomments}}{\textcolor{red}{(N.B: #1)}}{}}

\newcommand{\hbdry}{\partial_h}

\newcommand{\lampstand}[1]{
\draw [<->] (-#1.5,0) -- (#1.5,0);
}
\newcommand{\labellamps}[1]{
\foreach \x in {-#1,...,#1} \draw (\x,-.4) node [below] {\x};
}
\newcommand{\unlitlamps}[1]{
\foreach \x in {-#1,...,#1} \draw [fill=white] (\x,0) circle (4pt);
}
\newcommand{\lightlamps}[2]{
\foreach \x in {#1} \draw [fill=#2] (\x, 0) circle (3.5pt);
}
\newcommand{\llpath}[2]{
\draw [#2,|->] plot  coordinates {#1};
}
\newcommand{\llmark}[2]{
\draw (#1,.125) -- (#1, -.125) node [below] {#2};
}
\newcommand{\llpos}[1]{
\draw [<-,thick] (#1,.3) -- (#1, .5);
}
\newcommand{\lampbegin}
{\begin{center}\begin{tikzpicture}}
\newcommand{\lampend}{\end{tikzpicture}\end{center}}
\newcommand{\llpathcolor}{blue}
\newcommand{\llya}{.4}
\newcommand{\llyb}{.6}

\begin{document}

\renewcommand{\bf}{\bfseries}
\renewcommand{\sc}{\scshape}
\vspace{0.5in}

\title[Horofunction Boundary of Lamplighter Group $L_2$]{The Horofunction Boundary of the Lamplighter Group $L_2$ with the
Diestel-Leader Metric}

\author{Keith Jones}
\address{Department of Mathematics, Computer Science \& Statistics; State University of New York College at Oneonta; 108 Ravine Parkway; Oneonta, NY 13820}
\email{keith.jones@oneonta.edu}

\author{Gregory A. Kelsey}
\address{Department of Mathematics; Bellarmine University; 2001 Newburg Rd.; Louisville, KY 40205}
\email{gkelsey@bellarmine.edu}



\thanks{The authors thank Moon Duchin for 
helpful conversations and the anonymous referee for many thoughtful suggestions. 
Support for the first author from a summer research grant from
SUNY Oneonta is gratefully acknowledged. 
Both authors also thank the Institute for Advanced Study for its hospitality at the Park City Math Institute Summer Session 2012.}


\begin{abstract}
We fully describe the horofunction boundary $\hbdry L_2$ with the word metric
associated with the generating set $\{t,at\}$ (i.e the metric arising in the
Diestel-Leader graph $\dl(2,2)$). The visual boundary $\partial_\infty L_2$
with this metric is a subset of $\hbdry L_2$.  Although $\partial_\infty
L_2$ does not embed
continuously in $\partial_h L_2$, it naturally splits into two subspaces, each of which
is a punctured Cantor set and does embed continuously. 
The height function on $\dl(2,2)$ provides a natural
stratification of $\hbdry L_2$, in which countably-many non-Busemann points
interpolate between the two halves of $\partial_\infty L_2$.  Furthermore, the
height function and its negation are themselves non-Busemann horofunctions in
$\hbdry L_2$ and are global fixed points of the action of $L_2$.
\end{abstract}

\maketitle 

\section{Introduction}

The horofunction boundary $\partial_h X$ of a proper complete
 metric space $(X,d)$ is in general
defined as a subspace of the quotient of $C(X)$, the space of continuous
$\mathbb{R}$-valued functions on $X$, by constant functions
\cite[Definition II.8.12]{bh}. It suffices to
choose a base point $b$ in $X$ and use the embedding $i: X \hookrightarrow
C(X)$ sending $z\in X \mapsto d(z,x)-d(z,b)$.
Since $X$ is proper, the closure $\overline X$ of $i(X)$ in $C(X)$ provides
a compactification of $X$. 
We define $\partial_h X$ to be $\overline X \backslash i(X)$. 
We call a point in $\overline X$ a \emph{horofunction}, and given a sequence $(y_n)$
of points in $X$, one can define a horofunction associated to $(y_n)$ by

\begin{equation}
\label{horodefeqn}
h_{y_n}(x) = \lim_{n\rightarrow\infty} d(y_n,x) - d(y_n,b)
\end{equation}
provided this limit exists.

Gromov defines the horofunction boundary, which he calls the {\em ideal
boundary}, in the context of hyperbolic manifolds \cite{Gromov81}, but the
definition applies to any complete metric space. In \cite{bh} Bridson and
Haefliger use this construction in the context of CAT(0) spaces as a functorial
construction of the visual boundary.  The horofunction boundary also naturally
arises in the study of group $C^*$-algebras, where Rieffel, referring to it as
the {\em metric boundary}, demonstrates its usefulness
particularly in determining the $C^*$-algebra he calls the {\em cosphere
algebra}  \cite[\S3]{rieffel}.  

In this paper, $X$ is a group with a word metric, which is
$\mathbb{N}$-valued.\footnote{For us, $\mathbb{N}$ contains 0.} In this setting,
we define a geodesic ray to be an isometric
embedding $\mathbb{N} \rightarrow X$. We refer to point of $\partial_h X$ as a
{\em Busemann point} if it corresponds to a sequence of points lying along a
geodesic ray.  We will refer to the space of asymptotic classes of geodesic rays
in $(X,d)$ as the \textit{visual boundary} $\partial_\infty X$.  In CAT(0)
spaces, all horofunctions correspond to Busemann points; in fact, we can
extend $i$ to $\bar i: X \sqcup \partial_\infty X \rightarrow \overline X$, and
this is a homeomorphism \cite[\S II.8.13]{bh}.  In general one cannot expect an injective,
surjective, or even continuous map from $\partial_\infty X$ to $\partial_h X$.
Rieffel brings up the question of determining for a given space $(X,d)$ which
points of $\partial_h X$ are Busemann points \cite[after Definition
4.8]{rieffel}.  As an interesting example of non-injectivity, Reiffel
demonstrates that there are no non-Busemann points in $\partial_h \integers^n$
with the $\ell_1$ norm, and there are countably many Busemann points
\cite{rieffel}.  However, Kitzmiller and Rathbun demonstrate that
$\partial_\infty \integers^n$ is uncountable \cite{Z2boundary}.   

Others have studied the horofunction boundary of Cayley graphs of non-CAT(0) groups, 
often with variation in their terminology, though examples are still sparse.\footnote{Note,
for non-CAT(0) groups, $\partial_h$ depends on the generating set, as is
demonstrated in \cite[Example 5.2]{rieffel}.} 
Develin extended Rieffel's work to
abelian groups (he refers to the horofunction boundary as a \emph{Cayley
compactification} of the group) \cite{develin}. Friedland and Freitas found
explicit formulas for horofunctions for $GL(n, \mathbb{C})/\mathrm{U}_n$ with
Finsler $p$-metrics (they use the term \emph{Busemann compactification})
\cite{FF04}. Webster and Winchester (using the term \emph{metric boundary} as
Rieffel) studied the action of a word hyperbolic group on its horofunction
boundary and found it is amenable \cite{WW05}. They also established necessary
and sufficient conditions for an infinite graph to have non-Busemann points in
its horofunction boundary \cite{WW06}. Walsh has considered the horofunction
boundaries of Artin groups of dihedral type \cite{Walsh09} and the action of a
nilpotent group on its horofunction boundary \cite{Walsh11}. Klein and Nikas
have studied the horofunction boundary of the Heisenberg group equipped with
different metrics \cite{KN09}, \cite{KN10}. They determine the
isometry group of the Heisenberg group with the Carnot-Carath\'{e}odory metric.

The \textit{lamplighter group}, discussed more fully at the start of \S 2, 
is given by the presentation:
\[ L_2 = \langle a, t \mid a^2, [a^{t^i}, a^{t^j}] \forall\, i,j, \in
\integers\rangle \]
Let $S = \{t,at\}$. The generating set $S$ naturally arises when viewing the lamplighter
group as a group generated by a finite state automaton (FSA) \cite{GZ02}. 
This is a rare case where we are able to understand the Cayley graph of such a
group with its FSA generating set. 
In this case, the Cayley graph is the Diestel-Leader
graph $\dl(2,2)$ \cite{Woess-Lamplighter}.  In \cite{joneskelsey}, the authors
describe the visual boundary for Diestel-Leader graphs, which are
certain graphs arising from products of regular trees. When
there are more than two trees, the topology is indiscrete, 
but for two trees, the graph inherits enough
structure from its component trees that its visual boundary is an
interesting non-Hausdorff space. 
Since $\dl(2,2)$ (the product of two trees with valence 3) is a Cayley graph for
$L_2$, this provides a boundary for $L_2$ which is dependent on the generating set.
This boundary has a natural partition into two uncountable subsets, 
which we refer to as the upper and lower visual boundaries and 
denote by $\partial_\infty L_2^+$ and $\partial_\infty L_2^-$.
When equipped with the subspace topology, these subsets \emph{are} Hausdorff.

The goal of this paper is to fully describe $\partial_h L_2$ 
where the metric on $L_2$ is the word metric from $S$.
In Section \ref{distancesec} we provide some background on this metric, and
in Section \ref{visualembedssec} we discuss the relationship between
$\partial_\infty L_2$ and $\partial_h L_2$, proving

\begin{theorem*}[A - Corollary \ref{visualembedscor} and 
Observations  \ref{notcontinuousobs} and \ref{subspacesembedobs}]
There is a natural map $\partial_\infty L_2 \rightarrow \partial_h L_2$, which
is injective but not continuous.
When restricted to either $\partial_\infty L_2^+$ or $\partial_\infty L_2^-$,
however, this injection is continuous.
\end{theorem*}

In Section
\ref{modelssec}, we explicitly compute formulas for families of horofunctions,
including Busemann functions. 
It turns out the natural
height map $H: L_2 \rightarrow \integers$ (see Def.  \ref{HmMdef}) is a
non-Busemann horofunction.  
Section \ref{classificationsec} provides a proof that all of the points in
$\partial_h L_2$ are members of the families described in Section 4,
which is our main result.

\begin{theorem*}[B - Corollary \ref{bigcor}]
Every point in $\partial_h L_2$ belongs to one of the following families of horofunctions,
all of whose formulas we explicitly calculate in Section \ref{modelssec}.
\begin{itemize}
\item Busemann: These horofunctions arise from certain sequences of lamp stands 
where the union of positions of lit lamps is bounded below or above 
and the position of the lamplighter limits to positive or negative infinity.
\item spine: These horofunctions arise from certain sequences of lamp stands 
where the union of lit lamps is neither bounded below nor above 
and the position of the lamplighter limits to a finite value.
\item ribs: These horofunctions arise from certain sequences of lamp stands 
where the union of positions of lit lamps is bounded below or above but not both
and the position of the lamplighter limits to a finite value.
\item height: The natural height function and its negation arise as horofunctions from 
certain sequences of lamp stands where the union of lit lamps is neither bounded below nor above 
and the position of the lamplighter limits to positive or negative infinity.
\end{itemize}
The spine is parametrized by $\integers$ and the ribs by a subset of $L_2$, 
and so the set of non-Busemann horofunctions is countable.
\end{theorem*}

We describe the topology of $\partial_h L_2$ in Section \ref{topologysec} by 
determining the accumulation points, leading to the visualization 
in Figure \ref{boundarypic}.
 
\begin{figure}
\includegraphics[scale=.6]{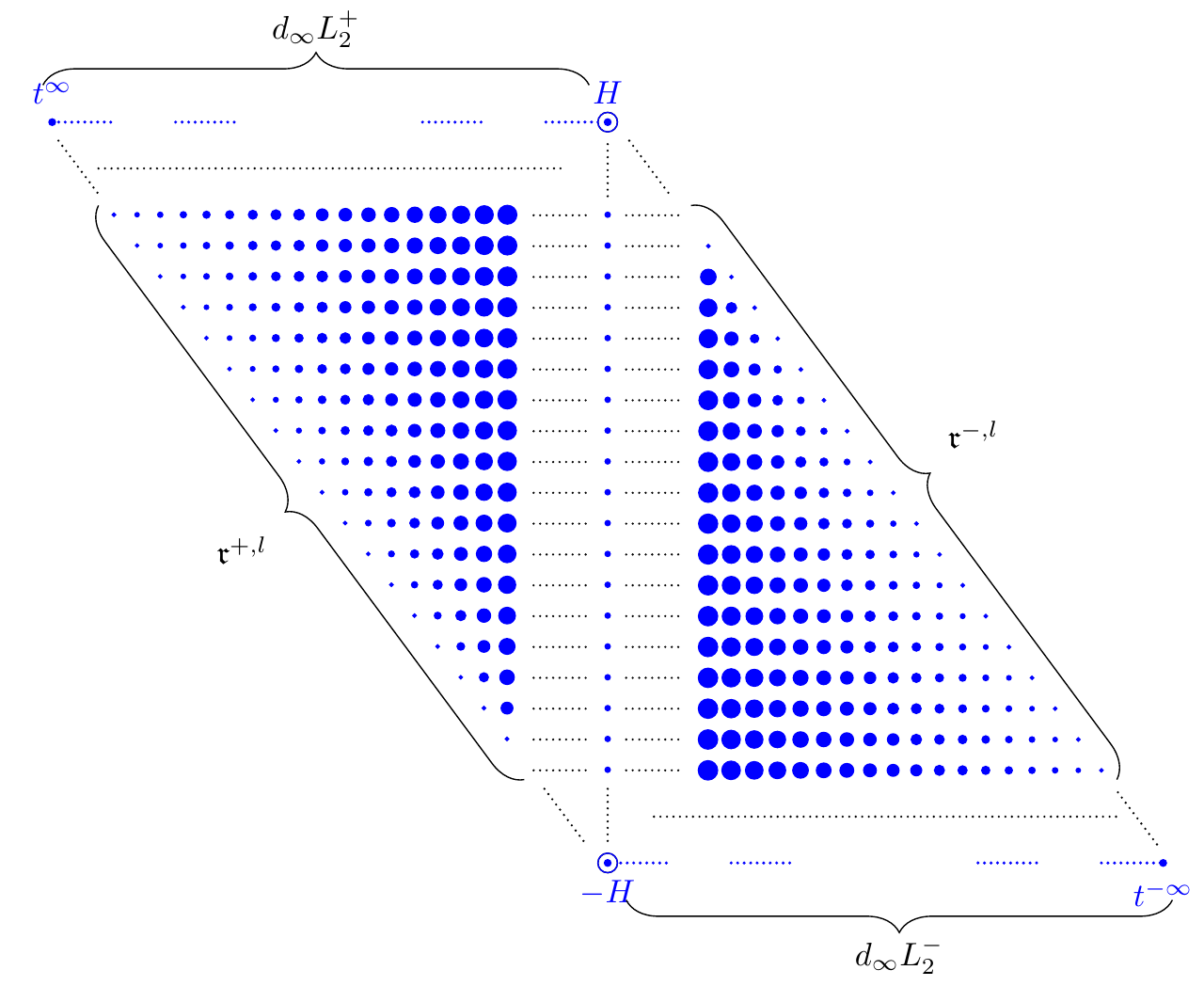}
\caption{Visualization of the horoboundary, including the spine (central
column), ribs (discrete point sets limiting to the corresponding spine point),
and upper and lower visual boundaries. For each rib, the dots of increasing size
represent finite discrete sets whose cardinalities double as we approach the
spine.\label{boundarypic}} 
\end{figure}

The names of the spine and ribs families come from the topology.
The spine family is parametrized by the limiting position of the lamplighter and appears in 
Figure \ref{boundarypic} as the central column of points.
For each spine function, there exist two subfamilies of ribs--
a ``positive rib'' and a ``negative rib''--
each a countable discrete subspace with the spine function
as its only accumulation point.
See the discussion in Section \ref{ribssubsec} for a thorough description
of these subfamilies.
 
Finally, Section \ref{actionsec} deals with some properties of the natural
action of $L_2$ on $\hbdry L_2$, in particular noting that $\pm H$ are global
fixed points.

\section{The Diestel-Leader metric on $L_2$}

\label{distancesec}

Let $d$ denote the word metric on $L_2$ with generating set $S=\{t, at\}$.
Since this is the metric on $L_2$ induced by the Cayley graph 
$\dl(2,2)$, we refer to it as the \emph{Diestel-Leader metric} on $L_2$.
Whenever we refer to $\partial_\infty L_2$ or $\hbdry L_2$, we always mean with $d$.
Stein and Taback have calculated the metric for general Diestel-Leader graphs
\cite{steintaback}, but in our case it is simple enough to review and provide a
proof.

Each element of $L_2$ is associated with a ``lamp stand'', which consists of an
infinite row of lamps in bijective correspondence with $\integers$, finitely
many of which are lit, and a marked lamp indicating the position of the
lamplighter. Figure \ref{lampstandfig} illustrates a typical example. 
The lamps are binary: either on or off. Right multiplying by
$a$ toggles the lamp at the lamplighter's position, while right
multiplying by $t$ increments the position of the lamplighter. We think of
this increment as a ``step right'' as in the figure. Using
$S$, the actions are either ``step (right or left)'' for
$t^{\pm1}$, ``toggle then step right'' for $at$, or ``step left then toggle''
for $(at)^{-1} = t^{-1}a$.

\begin{figure}
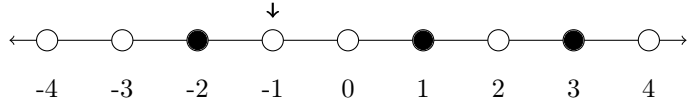

\lampbegin
\lampstand{4}
\labellamps{4}
\unlitlamps{4}
\lightlamps{-2,1,3}{black}
\llpos{-1}
\lampend
\caption{A typical element of $L_2$.} \label{lampstandfig}
\end{figure}

\begin{definition}
\label{HmMdef}
For $g\in L_2$, we define $H(g)$ to be the position of the lamplighter in the
lamp stand representing $g$, or equivalently the exponent sum of $t$ in a
word representing $g$, or the height of $g$ in $\dl(2,2)$. 

We define $m(g)$ to be equal to the minimum position of a lit lamp in the lamp stand representation of $g$ if the set of lit lamps is non-empty, and equal to $+\infty$ otherwise. 
Similarly, we define $M(g)$ to be equal to the maximum position of a lit lamp in the lamp stand representation of $g$ if the set of lit lamps is non-empty, and equal to $-\infty$ otherwise.

For $g_1, g_2 \in L_2$, we define $m(g_1, g_2)$ to be the minimum position of a lamp whose status differs in the lamp stands of $g_1$ and $g_2$ if such a position exists, and equal to $+\infty$ otherwise. Similarly, $M(g_1, g_2)$ is the maximum position of a lamp whose status differs in the lamp stands
of $g_1$ and $g_2$ if such a position exists, and is $-\infty$ otherwise.

We will define ``infinite lamp stands'' to represent boundary elements.
For these lamp stands, we define the $H$, $m$, and $M$ notation analogously.
\end{definition}

\begin{lemma}
\label{distancelem}

If $g_1, g_2\in L_2$, then $d(g_1, g_2)=2(B-A)-C$ where
\begin{itemize}
\item $A = \min\{m(g_1,g_2), H(g_1), H(g_2)\}$ is the left-most position the
lamplighter must visit to change between $g_1$ and $g_2$,
\item $B = \max\{M(g_1,g_2)+1, H(g_1), H(g_2)\}$ is the right-most such
position, and
\item $C = |H(g_2) - H(g_1)|$ is the distance between the lamplighter's
positions in $g_1$ and $g_2$
\end{itemize}
\end{lemma}
See Figure \ref{pathfig} for an illustration of a typical path.

\begin{figure}
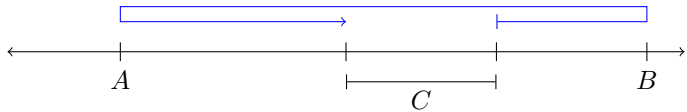

\vspace{.125in}
\lampbegin
\lampstand{4}
\llmark{-3}{$A$}
\llmark{0}{$$}
\llmark{2}{$$}
\llmark{4}{$B$}
\draw [|-|] (0, -\llya) --  node [below] {$C$} (2,-\llya);
\llpath{(2,\llya) (4,\llya) (4,\llyb) (-3,\llyb) (-3,\llya)
(0,\llya)}{\llpathcolor}
\lampend
\caption{Distance between two elements of $L_2$ with Diestel-Leader metric.
}
\label{pathfig}
\end{figure}

\begin{proof}
Since the Cayley graph is vertex transitive, without loss of generality we may
assume that $g_1=id$ and we denote $g_2$ simply by $g$. We will consider a
geodesic from $id$ to $g$ on the lamp stand representations of the elements of
$L_2$.

A geodesic will start at $id$ with no lit lamps and the lamplighter at position
$H(id) = 0$. The lamplighter must move in one direction (either left or right)
until it has gone as far as it needs to, it then travels to the other extremal
position, and then finishes by moving to $H(g)$. The initial direction will be \emph{away} from $H(g)$ in order to minimize the total distance. Notice
that the minimum extremal position is given by $A$, which in this case is
$A=\min\{m(g), H(g), 0\}$, and the maximal extremal position is given by
$B$, which in this case is $B=\max\{M(g)+1, H(g), 0\}$. Notice that we use
$M(g)+1$ and not $M(g)$ since to turn on the lamp at position $k$, the lamplighter
 must be at position $k+1$ either immediately before turning on lamp $k$
(if using generator $(at)^{-1}$) or immediately after (if using generator
$at$).

Thus, the second of the three segments of the geodesic will have length $B-A$.
The lengths of the first and third segments will sum to less than $B-A$, and
the amount less will be exactly equal to the distance between the starting and
ending position, which in our case is $|H(g)|$.
\end{proof}

\section{Busemann points}

\label{visualembedssec}

\subsection{The visual boundary}
\label{busemannlampsec}

As in \cite[Section 3.3]{joneskelsey}, we can interpret elements of the visual
boundary in terms of the lamp stand model. Such an element can be represented
by a geodesic ray emanating from the identity which follows a sequence of steps
wherein the lamplighter first moves one direction until reaching the extremal
lit lamp in that direction 
then ``turns around''
and marches off towards $\pm \infty$ toggling lamps as necessary. Thus, in the limit there
is either a minimal lit lamp (if any are lit at all), and the lamplighter is at
$+\infty$; or there is a
maximal lit lamp (if any are lit at all), and the lamplighter is at $-\infty$.
A ``turning around'' only occurs if the minimal lit lamp has negative index in the
former case, or the maximal lit lamp has positive index in the latter. 
This final configuration of lit lamps gives an ``infinite lamp stand'' for the geodesic ray.

In \cite[Observations 4.10 and 4.11]{joneskelsey}
the authors investigate the visual boundary of 
$\dl(2,2)$ and find that as a set, it is a disjoint union of the
sets $\partial_\infty L_2^+$ and $\partial_\infty L_2^-$,
where
$\partial_\infty L_2^\pm$ is the set of those asymptotic classes 
with lamplighter at $\pm\infty$.
These two sets both have the subset topology of punctured Cantor sets,
but the full visual boundary is not Hausdorff. 
We provide the intuition here.

By \cite[Lemma 3.5]{joneskelsey} in $L_2$ geodesic rays that are asymptotic eventually merge.
For example, if a ray has the lamplighter go from $0$ to $-n$ and then in 
the positive direction forever, the lamps from $-n$ to $0$ will be traversed twice.
Therefore, the initial setting of lamps on the first pass can be re-done on the
second pass.
The asymptotic class of the ray includes all the different initial settings that become
the same final setting when the lamplighter moves in its final direction.
Thus, the infinite lamp stand of a ray is actually an invariant of its asymptotic class.

Notice that such a ray that has the lamplighter go from $0$ to $-n$ and then in 
the positive direction forever is in $\partial_\infty L_2^+$, but is close
in the visual boundary topology to rays in $\partial_\infty L_2^-$ that have 
the lamplighter only move in the negative direction and agree on the initial
settings of the lamps $0$ through $-n$.
The fact that these initial settings can be made arbitrary within the
asymptotic equivalence class
gives us a large subset of $\partial_\infty L_2^-$ that is contained in a neighborhood
of \emph{any} ray in $\partial_\infty L_2^+$ where the lamplighter moves in the
negative direction for a long time before eventually moving in the positive direction
forever.

Thus, there exist distinct elements of $\partial_\infty L_2^+$ 
whose neighborhoods always intersect,
and that intersection is a subset of $\partial_\infty L_2^-$.
Therefore $\partial_\infty L_2$ is not Hausdorff.
Recall that both $\partial_\infty L_2^\pm$ are punctured 
Cantor sets under the subspace topology.
So, while the subspace topologies of these ``halves''
are Hausdorff, they are not compact.
The full visual boundary $\partial_\infty L_2$ is, however,
compact, since these troublesome open sets
that intersect both $\partial_\infty L_2^\pm$ ``fill''
the punctures with open sets in the opposite half.

\subsection{The visual boundary as a subset of the horoboundary}

We now show that there is a natural injection from the non-Hausdorff
$\partial_\infty L_2 =\partial_{\infty} \dl(2,2)$ into 
$\partial_h L_2$.
Since $\partial_h L_2$ is Hausdorff, this injection is non-continuous.

\begin{lemma}[Lemma 8.18(1) in Chapter II.8 of \cite{bh}]

Let $\gamma$ be a geodesic ray in $\dl(2,2)$ based at the identity.
Then the sequence of points $(\gamma(n))$ defines a horofunction $\mathfrak b^\gamma$.

\end{lemma}

The horofunction $\mathfrak b^\gamma$ is called the \emph{Busemann function} associated to
$\gamma$. In a CAT(0) space, the Busemann functions of two rays are equal if and
only if those two rays are asymptotic. Even though $\dl(2,2)$ is not CAT(0), the same is true in our case. 

\begin{lemma}
\label{visualembedslemma}
Let $\gamma, \gamma'$ be geodesic rays in the Cayley graph of $L_2$ based at the identity. 
The Busemann functions $\mathfrak b^\gamma$ and $\mathfrak b^{\gamma'}$ are equal if and only if $\gamma$ and $\gamma'$ are asymptotic to each other.

\end{lemma}

\begin{proof}
Recall that asymptotic rays in $\dl(2,2)$ eventually merge. 
Thus, the Busemann functions of asymptotic rays are equal.

Now suppose that $\gamma$ and $\gamma'$ are \emph{not} asymptotic to each other.  Let
$\alpha\in[\gamma], \alpha'\in[\gamma']$ (i.e. $\alpha$ is in the asymptotic equivalence class of $\gamma$) so that $\alpha$ and $\alpha'$ have
maximal shared initial segment.  Say that this shared initial segment has length
$k$.  Let $x=\alpha(k+1)$.  Notice that by definition, $\mathfrak b^\alpha(x) = -(k+1)$.
By our choice of $\alpha$ and $\alpha'$, $\mathfrak b^{\alpha'}(x) = -(k-1)$, so
$\mathfrak b^\alpha
\neq \mathfrak b^{\alpha'}$.  By the proof above of the other direction,
$\mathfrak b^\gamma=\mathfrak b^\alpha$ and $\mathfrak b^{\gamma'}=\mathfrak b^{\alpha'}$ and we are done.  
\end{proof}

\begin{corollary}
\label{visualembedscor}
The relation taking an asymptotic equivalence class of geodesic rays based at
the identity to
their Busemann functions is an injection of $\partial_\infty L_2$ into
$\partial_h L_2$.
\end{corollary}

\begin{observation}
\label{notcontinuousobs}

The injection in Corollary \ref{visualembedscor} is not continuous.

\end{observation}

\begin{proof}
The continuous injective image of a non-Hausdorff space like $\partial_\infty L_2$ must also be non-Hausdorff, while $C(L_2)$ (and thus its subspace $\hbdry L_2$) is Hausdorff.
\end{proof}

Recall that the non-Hausdorff property was proved by finding neighborhoods of
distinct elements of $\partial_\infty L_2^+$ that always shared elements of
$\partial_\infty L_2^-$.
Observation \ref{subspacesembedobs} shows that the restriction of
this injection to either of the subspaces of the visual boundary $\partial_\infty L_2^\pm$
is continuous.

\section{Model horofunctions}

\newcommand{\ignore}[1]{}
\label{modelssec}
In this section, we construct four families of ``model'' horofunctions, and in
\S\ref{classificationsec} we show that these represent all horofunctions.

We break $\partial_h L_2$ into four categories: the Busemann points, the
\textit{spine}, the \textit{ribs}, and the two points $\pm H$.  The
reader may refer to Figure \ref{boundarypic} on page \pageref{boundarypic}
to preview a visualization of the boundary, illustrating our choice of terms.
To determine which category a sequence $(x_n)$ in $L_2$ falls into (if it
defines a horofunction at all), it turns out we need only consider 
whether $H(x_n)$ approaches an integer or $\pm\infty$, and whether the union
over all lit lamps in the sequence is bounded above or below. 

\subsection{The Spine}
\label{spinesubsec}

Fix $l\in\integers$, and let $(s_n^l), n\in \mathbb{N}$, be the sequence in
$L_2$ having lamps at $\pm n$ lit, all others unlit, and $H(s_n^l) =
l$.

Applying Lemma \ref{distancelem} with $A = -n$, $B = n+1$, and $C = |l|$,
we have  $d(s_n^l, id) = 4n + 2 - |l|$. 
Given any $g \in L_2$, take $n > \max\{-m(g), M(g), |H(g)|, |l|\}$ and 
apply Lemma \ref{distancelem} to obtain $d(s_n^l, g) = 4n + 2 - |l - H(g)|$.
By Equation \ref{horodefeqn}, the horofunction is
\begin{equation}
\label{spineeqn}
\mathfrak{s}^l(g)  = h_{s_n^l}(g) = |l| - |l - H(g)|. 
\end{equation}

We call this the \textit{spine horofunction at height} $l$. For a given $l$,
this is a function of only $H(g)$.
Figure \ref{hfigs} shows the graphs of $\mathfrak{s}^{-1}$, $\mathfrak{s}^{0}$,
$\mathfrak{s}^1$, and $\mathfrak{s}^{2}$, respectively, as functions of height.
The spine horofunction at height 0 is $\mathfrak{s}^0 = -|H(g)|$. One can check
that the sequence $(s_n^n)$ yields $H(g)$ and $(s_n^{-n})$ yields
$-H(g)$; and we can see that the spine functions interpolate between between the
two. 

\begin{figure}
\begin{subfigure}[t]{.5\linewidth}
\centering
\scalebox{.6}{
\begin{tikzpicture}
\draw [gray,very thin,step=.5] (-4,-2) grid (4,2);
\draw [<->] (-4.1,0) -- (4.1,0);
\draw [<->] (0,-2.1) -- (0,2.1);
\foreach \x in {-4,...,4} \draw (\x,2pt) -- (\x,-2pt) node [below] {$\scriptstyle
\x$};
\draw [blue, domain=-4:2,samples=400] plot (\x, {1-abs(\x+1)});
\end{tikzpicture}}
\caption{$\mathfrak{s}^{-1}$.} \label{s^-1fig}
\end{subfigure}%
\centering
\begin{subfigure}[t]{.5\linewidth}
\centering
\scalebox{.6}{
\begin{tikzpicture}
\draw [gray,very thin,step=.5] (-4,-2) grid (4,2);
\draw [<->] (-4.1,0) -- (4.1,0);
\draw [<->] (0,-2.1) -- (0,2.1);
\foreach \x in {-4,...,4} \draw (\x,2pt) -- (\x,-2pt) node [below] {$\scriptstyle
\x$};
\draw [blue, domain=-2:2,samples=400] plot (\x, {-abs(\x)});
\end{tikzpicture}}
\caption{$\mathfrak{s}^0$.} \label{s^0fig}
\end{subfigure}

\begin{subfigure}[t]{.5\linewidth}
\centering
\scalebox{.6}{
\begin{tikzpicture}
\draw [gray,very thin,step=.5] (-4,-2) grid (4,2);
\draw [<->] (-4.1,0) -- (4.1,0);
\draw [<->] (0,-2.1) -- (0,2.1);
\foreach \x in {-4,...,4} \draw (\x,2pt) -- (\x,-2pt) node [below] {$\scriptstyle
\x$};
\draw [blue, domain=-2:4,samples=400] plot (\x, {1-abs(\x-1)});
\end{tikzpicture}}
\caption{$\mathfrak{s}^1$.} \label{s^1fig}
\end{subfigure}%
\begin{subfigure}[t]{.5\linewidth}
\centering
\scalebox{.6}{
\begin{tikzpicture}
\draw [gray,very thin,step=.5] (-4,-2) grid (4,2);
\draw [<->] (-4.1,0) -- (4.1,0);
\draw [<->] (0,-2.1) -- (0,2.1);
\foreach \x in {-4,...,4} \draw (\x,2pt) -- (\x,-2pt) node [below] {$\scriptstyle
\x$};
\draw [blue, domain=-2:4,samples=400] plot (\x, {2-abs(\x-2)});
\end{tikzpicture}}
\caption{$\mathfrak{s}^2$.} \label{s^2fig}
\end{subfigure}
\caption{Four ``spinal'' horofunctions, as functions of height of $g\in L_2$.}
\label{hfigs}
\end{figure}
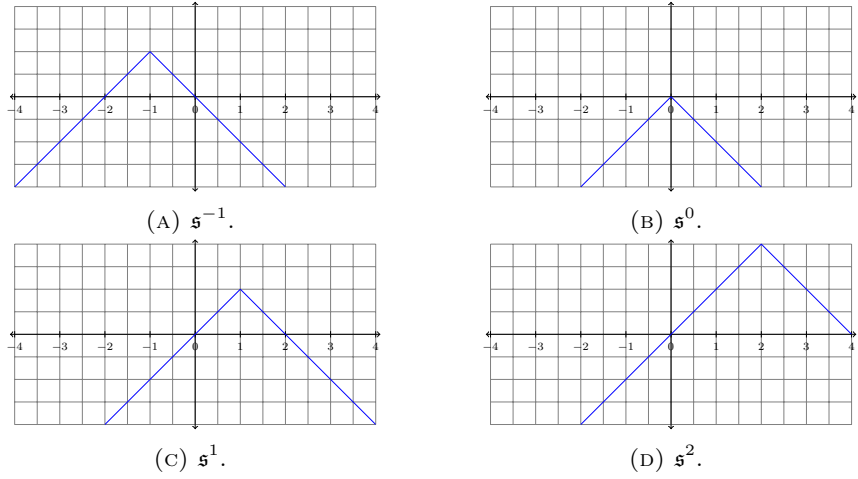

\subsection{The Ribs}
\label{ribssubsec}

The rib horofunctions will be parameterized by certain elements of $L_2$. There
are two subfamilies, corresponding to  the $+\infty$ and $-\infty$ direction, and the
generating set $\{t,at\}$ creates a slight asymmetry between them. Let $f \in
L_2$, and set $l = H(f)$.

First, assume $M(f) < l$ (noting that $M(f)$ may equal
$-\infty)$. Then consider the sequence $(r^{+,f}_n),\ n\geq
l$, in $L_2$ where the lamps of $r^{+,f}_n$ agree with those of $f$ in each
position below $l$, $H(r^{+,f}_n) = l$, lamp $n$ is lit, and no other
lamps at positions $l$ or above are lit.  

Given $g\in L_2$, take $n$ large enough, and we have:

\begin{align*}
d(r^{+,f}_n, id) & = 2((n+1) - \min\{m(f),l,0\})  - |l|\\
d(r^{+,f}_n, g)  & = 2((n+1) - \min\{m(f,g),H(g),l\}) - |l-H(g)| 
\end{align*}

This yields the \textit{(positive) rib horofunction} corresponding to $f$:

\begin{align}
\label{posribeqn}
\mathfrak r^{+,f}(g) & = 2(\min\{m(f),l,0\} - \min\{m(f,g),H(g), l\}) - |l -
H(g)| + |l|  \nonumber \\
& = 2(\min\{m(f),l,0\} - \min\{m(f,g), H(g), l\}) + \mathfrak s^l(g) 
\end{align}

We can see that if we had chosen an element whose lamps agreed with $f$ below
$l$, but also had lamps in position $l$ or higher lit, defining a
sequence similarly would
lead us to the same horofunction, since we can always toggle lamps at $l$ or
above ``for free'' with the generator $at$. 

Though the set of positive rib horofunctions is discrete, there is some
structure to be observed.  Given a height $l$, let $R^+_l$ be the set of
positive rib horofunctions at height $l$. Each corresponds to an element $f \in
L_2$ with $H(f) = l$ and $M(f) < l$. Then the ``minimum lit lamp'' map $m: L_2
\rightarrow \overline{\integers}$ induces a map  $\hat m_l: R^+_l \rightarrow
\overline{\integers}$.  For $k \in \overline{\integers}$, the cardinality of
$\hat m_l^{-1}(k)$ is $2^{(l-k-1)}$ if $k<l$, 1 if $k = +\infty$, and 0
otherwise.  The set $R^+_l$ can then be partitioned according to the nonempty
preimages, which provides a natural filtration of $R^+_l$.  Any sequence $(r_n)$ of
horofunctions in $R^+_l$ corresponding to a sequence $(f_n)$ in $L_2$ with
$m(f_n) \rightarrow -\infty$, will approach $\mathfrak s^l$.  We make a precise
argument for this fact in Observation \ref{spineaccumofribsobs}.  

In the special case that $f$ has no lit lamps, then $M(f) = -\infty$ and $m(f) =
+\infty$ and the calculation simplifies. Since $m(f,g) = m(g)$, the
only data is the height $l= H(f)$; and we have:

\begin{align}
\mathfrak r^{+,f}(g) = \mathfrak r^{+,l}(g) 
& = 2(\min\{l,0\} - \min\{m(g),H(g),l\}) - |l - H(g)| + |l|\nonumber \\ 
& = -2\min\{m(g),H(g),l\} - |l - H(g)| + l
\end{align}

As indicated in the preceding paragraph, when there are no lit lamps the
resulting horofunction $\mathfrak r^{+,l}$ is in a sense the farthest positive
rib function of height $l$ from the spine, and we think of it as the \textit{rib
tip} at height $l$.  

We now turn to the negative rib functions, corresponding to those $f$ that
satisfy $m(f) \geq l$ (possibly with
$m(f) = +\infty)$. Note we use ``$\geq$'' now, whereas we used ``$<$'' previously, since
the status of the lamp at $l$ does matter in this direction, since
 using $(at)^{-1}$ will only let us toggle lamps in positions
$l-1$ or lower ``for free''. One can define a corresponding sequence similarly
to the positive direction, except that the lit lamps approach $-\infty$, and
calculate the horofunction:

\begin{align}
\label{negribeqn}
\mathfrak r^{-,f}(g) & = 2(\max\{M(f,g)+1, H(g), l\} - \max\{M(f)+1, l, 0\}) -
|l - H(g)| + |l|\nonumber \\
& = 2(\max\{M(f,g)+1, H(g), l\} - \max\{M(f)+1, l, 0\}) + \mathfrak s^l(g) 
\end{align}

There is a similar simplification in this direction when $f$ has no lit lamps,
so that the horofunction depends only on $l$:

\begin{align}
\mathfrak r^{-,f}(g) = \mathfrak r^{-,l}(g) 
& = 2(\max\{M(g)+1, H(g), l\} - \max\{l,0\}) - |l - H(g)| + |l|\nonumber \\
& = 2\max\{M(g) + 1, H(g), l \} - |l-H(g)| - l 
\end{align}

Finally, the set $R^-_l$ of negative rib horofunctions at height $l$ has a
structure similar to $R^+_l$.

\subsection{Busemann Functions}
Given a geodesic ray $\gamma$ with $\gamma(0)=id$, let $\mathfrak b^\gamma$
denote its horofunction. Let $g\in L_2$.  As discussed in Definition
\ref{HmMdef} and \S \ref{busemannlampsec}, we can define the functions $m$ and
$M$ similarly for $\gamma$. We either have $\gamma \in \partial_\infty L_2^+$
and $m(\gamma)$ and $m(\gamma,g)$ are defined, or $\gamma \in \partial_\infty
L_2^-$ and $M(\gamma)$ and $M(\gamma,g)$ are defined.  The formula for
$\mathfrak b^\gamma$ depends on the direction of $\gamma$, so we use $\mathfrak
b^{+,\gamma} = \mathfrak b^\gamma$ when $\gamma \in \partial_\infty L_2^+$ and
$\mathfrak b^{-,\gamma} = \mathfrak b^\gamma$ when $\gamma \in \partial_\infty
L_2^-$, to be clear.  

When $\gamma \in \partial_\infty L_2^+$, for $n$ large
enough, we apply Lemma \ref{distancelem} to obtain:
\begin{align*}
d(\gamma(n),id)&  = 2(H(\gamma(n))-\min\{m(\gamma),0\}) - H(\gamma(n)) \\
d(\gamma(n),g)&  = 2(H(\gamma(n))- \min\{m(\gamma,g),H(g)\}) + H(g) - H(\gamma(n))
\end{align*}
Thus the Busemann function corresponding to $\gamma$ is given by
\begin{equation}
\label{busemannposeqn}
\mathfrak b^{+,\gamma}(g)  = 
2\left(\min\{m(\gamma),0\} - \min\{m(\gamma,g),H(g)\}\right)+H(g)
\end{equation}
If $\gamma \in \partial_\infty L_2^-$, we can similarly calculate
\begin{equation}
\label{busemannnegeqn}
\mathfrak b^{-,\gamma}(g)=2\left(\max\{M(\gamma,g)+1,H(g)\} -
\max\{M(\gamma)+1, 0\}\right)-H(g)
\end{equation}

Note that the Busemann horofunctions are obtained from the rib
horofunctions by allowing the lamplighter position to approach $+\infty$ or
$-\infty$ as appropriate.  
This is spelled out later in Observation \ref{busemannaccumofribsobs}.

Given any two horofunctions described above, one can find an element $g$ 
of $L_2$ on
which they disagree. Thus we have the following observation.
\begin{observation}
\label{nonbusemannobs}
The horofunctions $\mathfrak s^l$ for $l \in \integers$, $\pm H$, $\mathfrak
r^{+,f}$ for $f \in L_2$ and $M(f) < H(f)$, $\mathfrak r^{-,f}$ for $f \in L_2$
and $m(f) \geq H(f)$, $\mathfrak b^{+,\gamma}$, $\mathfrak b^{-,\gamma}$,
$\gamma \in \partial_\infty L_2$, are all pairwise distinct.
\end{observation}

\section{Classification of horofunctions}

\label{classificationsec}

We will now prove that the functions referred to in Observation
\ref{nonbusemannobs} constitute all of $\partial_h L_2$.

\begin{definition}

Given a sequence $(g_n)\subset L_2$, we say that the lamp at position $k$ in the lamp stands of these elements \emph{stabilizes} if there exists $N\in\naturals $ such that the lamp in position $k$ for the lamp stand representing $g_n$ has the same status (i.e. on or off) for all $n>N$.

We say that the lamp at position $k$ is \emph{flickering} if it does not stabilize.

\end{definition}

\begin{definition}

We say that sequence $(g_n)$ of elements of $L_2$ is \emph{right stable} if there
exists $N\in\naturals$ and
$M\in\integers$ such that for all $k>M$, the lamp at position $k$ for the lamp
stand representing $g_n$ has the same status (i.e. on or off) for all $n>N$.
That is, a sequence is right stable if the set of positions of its flickering lamps 
(should any exist) has a maximum.

We define \emph{left stable} similarly.

\end{definition}

\begin{observation}
\label{noboundsubseqobs}

If a sequence $(g_n)\subset L_2$ is not right stable, 
then there exists a subsequence $(g_{n_k})$ such that
the sequence $(M(g_{n_k}))$ is increasing
without bound.

Similarly, if a sequence $(g_n)\subset L_2$ is not left stable,
then there exists a subsequence $(g_{n_k})$ such that the sequence $(m(g_{n_k}))$ is
decreasing without bound.

\end{observation}

\begin{proof}

If the sequence is not right stable, then $\sup\{M(g_n)\ |\
n\in\naturals\} = +\infty$ since if this supremum were a finite value
$M_0\in\integers$, then by setting $N=0$ and $M=M_0$, the sequence would satisfy
the definition for being right stable. The existence of the desired subsequence is then guaranteed.

The proof when the sequence is not left stable is similar.
\end{proof}

\begin{lemma}
\label{noflickeringbelowlem}

Suppose that a sequence $(g_{n})\subset L_2$ with $H(g_n) \to l \in
\integers\cup\{+\infty\}$ is left stable. 
If $(g_n)$ is associated with some horofunction $h_{g_n}$, then 
the set of positions of its flickering lamps (should any exist) 
has a minimum of at least $l$.

\end{lemma}

\begin{proof}

If $(g_n)$ has no flickering lamps, then we are done. 
So assume the sequence has some flickering lamps, and let
$k\in\integers$ be the minimum position of a flickering lamp. 
Suppose for contradiction that $k<l$.

Let $y\in L_2$ such that $H(y) = k$, $y$ agrees with the stabilization of lamps of $(g_n)$ on the positions $k-1$ and below, and the lamp at position $k$ is off. Let $x\in L_2$ be exactly as $y$, except that $H(x) = k+1$. Let $n$ be sufficiently large so that the lamps at positions $k-1$ and below of $g_n$ have achieved their eventual status and $H(g_n)>k$.

Suppose the lamp at position $k$ is lit in the lamp stand for $g_n$.
In Lemma \ref{distancelem}, when computing 
$d(g_n,x)$, $C=H(g_n)-(k+1)$, but when computing $d(g_n,y)$, 
$C=H(g_n)-k$, while the values for $A$ and $B$ remain the same (in this
case, $A=k$ for both).  Thus $d(g_n,x) = d(g_n,y)+1$.

Now suppose the lamp at position $k$ is \emph{not} lit in the lamp stand for $g_n$.  
Using Lemma \ref{distancelem} again, when
computing $d(g_n, x)$, $A = k+1$, $C=H(g_n)-(k+1)$, while when computing $d(g_n,
y)$, $A = k$, $C=H(g_n)-k$, and $B$ remains the same.
In this case, we have $d(g_n,x) = d(g_n,y)-1$.

By Equation \ref{horodefeqn},
$$h_{g_n}(x) -h_{g_n}(y) = \lim_{n\to\infty} d(g_n, x) - d(g_n, y)$$
which by the above, does not exist. But we assumed $h_{g_n}$ exists.
Hence, our assumption that $k<l$ is incorrect, and we have the
desired result. 
\end{proof}

\begin{lemma}
\label{noflickeringabovelem}

Suppose that a sequence $(g_n)\subset L_2$ with $H(g_n) \to l \in
\{-\infty\}\cup\integers$ is right stable.
If $(g_n)$ is associated with some horofunction $h=h_{g_n}$, then for every
$k\geq l$, the lamp at position $k$ stabilizes.

\end{lemma}

\begin{proof}

The proof for this lemma is the same as for Lemma \ref{noflickeringbelowlem}. The asymmetry in the inequalities (one is strict, while the other is not) comes from the asymmetry of our generating set (including $at$ but not $ta$).
\end{proof}

\begin{lemma}
\label{bothboundslem}

Suppose that a sequence $(g_n)\subset L_2$ is both left and right stable
and that $H(g_n) \to l\in\integers$. If $(g_n)$ is associated with some
horofunction $h_{g_n}$, then there is $g \in L_2$ such that $g_n \rightarrow g$
(i.e. the sequence is eventually constant), and $h_{g_n}$ is associated to the
image of $g$ in $\overline{L_2}$.

\end{lemma}

\begin{proof}
By Lemmas \ref{noflickeringbelowlem} and \ref{noflickeringabovelem}, all the
lamps in $(g_n)$ stabilize. Since it is stable on both sides, we in fact have
the existence of some $N\in\naturals$
such that the set of lit lamps in $g_n$ is constant for all $n>N$. Since the
lamplighter limits to $l$ by hypothesis and since $\integers$ is a discrete set, we have that the sequence $(g_n)$ is eventually constant.
\end{proof}

\begin{lemma}
\label{oneboundlem}

Suppose that a sequence $(g_n)\subset L_2$ is either left or right stable,
 \emph{but not both}, and that $H(g_n) \to
 l\in\integers$. If $(g_n)$ is associated with some horofunction
$h_{g_n}$, then $h_{g_n}$ is a rib, i.e one of 
$\mathfrak r^{\pm,f}$, $f\in L_2$.

\end{lemma}

\begin{proof}

We consider the case where the sequence $(g_n)$ is left stable, but not right
stable.  The other case is similar.

By Lemma \ref{noflickeringbelowlem}, there exists $N\in\naturals$ such that the
the lamps below position $l$ are stable and $H(g_n) = l$ for all $n>N$.  Let
$\mathfrak{r}$ be the rib horofunction that matches this stabilization. Set
$(r_n)$ to be the model sequence defined in Section \ref{ribssubsec} that generates this horofunction. 

By Observation \ref{noboundsubseqobs}, we may take a subsequence $(g_{n_k})$ such that $(M(g_{n_k}))$ is increasing with $M(g_{n_k})> k$ for all $k$. Choose a subsequence $(r_{n_k})$ of our model sequence such that $M(r_{n_k}) = M(g_{n_k})$.

Let $x\in L_2$. Choose $K\in\naturals$ such that $K> \max\{|l|,\ |M(x)|,\ |H(x)|\}$, and let $k>K$. 

Let $A,B,C$ be as in Lemma \ref{distancelem} for the computation of $d(g_{n_k},
x)$ and let $A',B',C'$ be as in Lemma \ref{distancelem} for the computation of
$d(r_{n_k}, x)$. Notice that $A=A'$ since the lamp stands for $g_{n_k}$ and
$r_{n_k}$ are the same below the position $H(g_{n_k}) = H(r_{n_k})$, $B =
M(g_{n_k})+1 = M(r_{n_k})+1=B'$ by our choice of $K$, and $C=C'$ since $H(g_{n_k}) =
H(r_{n_k})$. Thus, $d(g_{n_k}, x)=d(r_{n_k}, x)$.

For $x=id$, we have that $d(g_{n_k}, id) = d(r_{n_k}, id)$. Hence, $h_{g_{n_k}} = h_{r_{n_k}}$ and so therefore $h_{g_n} = \mathfrak{r}$.
\end{proof}

\begin{lemma}
\label{noboundslem}
Suppose that a sequence $(g_n)\subset L_2$ is neither left nor right stable
 and that $H(g_n) \to l\in\integers$.
If $(g_n)$ is associated with some horofunction $h_{g_n}$, then
$h_{g_n}=\mathfrak{s}^l$.

\end{lemma}

\begin{proof}

Suppose that there exists a subsequence
$(g_{n_k})$ such that for all $N\in\naturals$ there exists $K_N\in\naturals$
such that for all $k>K_N$ we have that $M(g_{n_k})>N$ and $m(g_{n_k})<-N$. Then
let $x\in L_2$, and let $N\in\naturals$ such that $N>\max\{|M(x)|,\ |m(x)|,\
|l|\}$. Let $K=\max\{K_N, K_l\}$ where $K_N$ is as given above and $K_l$ is an
integer such that for all $k>K_l$, $H(g_{n_k}) =l$ (recall that 
$H(g_n) \to l$ and the integers are a discrete set).

Let $k>K$. Then by choice of $N$ and definition of $K$ and using Lemma
\ref{distancelem}, $d(g_{n_k}, x) = 2(M(g_{n_k})+1 - m(g_{n_k}))-|l-H(x)|$ and
specifically $d(g_{n_k}, id) = 2(M(g_{n_k})+1 - m(g_{n_k}))-|l|$. Thus, 
by Equation \ref{horodefeqn} $h_{g_{n_k}} = \mathfrak{s}^l(x)$, and we are done.

Now suppose that such a subsequence does \emph{not} exist. By Observation
\ref{noboundsubseqobs}, since $(g_n)$ is not left stable, there exists a subsequence $(g_{n_i})$ such that
$m(g_{n_i})<-i$ for all $i$ and $(m(g_{n_i}))$ is decreasing. Also by
Observation \ref{noboundsubseqobs}, since $(g_n)$ is not right stable, there exists a subsequence $(g_{n_j})$ such that
$M(g_{n_j})>j$ for all $j$ and $(M(g_{n_j}))$ is increasing. Since these are
both subsequences of $(g_n)$, both give rise to horofunctions, and $h_{g_{n_i}}
= h_{g_{n_j}} = h_{g_{n}}$.

Notice that the subsequence $(g_{n_i})$ must be right stable, otherwise we
would be able to find a subsequence as in the first part of the proof.
Similarly, the subsequence $(g_{n_j})$ must be left stable.

By Lemma \ref{oneboundlem}, $h_{g_{n_i}}$ is equal to one of the rib examples
with stable component above the lamplighter. But also by Lemma
\ref{oneboundlem}, $h_{g_{n_j}}$ is equal to one of the rib examples with stable
component \emph{below} the lamplighter. By inspecting Equations
\ref{posribeqn} and \ref{negribeqn}, we
see that these two horofunctions cannot be equal, so $h_{g_{n}}$ does not
exist.
\end{proof}

\begin{lemma}
\label{busemannposlem}

Suppose that a sequence $(g_n)\subset L_2$ is left stable and $H(g_n) \to +\infty$. If $(g_n)$ is associated
with some horofunction $h_{g_n}$, then $h_{g_n}$ is equal to a Busemann function
$\mathfrak b^\gamma$ with $[\gamma]\in\partial_\infty L_2^+$.
\end{lemma}

\begin{proof}

By Lemma \ref{noflickeringbelowlem}, there are no flickering lamps in $(g_n)$,
so consider the infinite lamp stand of the stabilization of lamps in $(g_n)$.
Since the sequence is left stable, if there are any lamps lit in this infinite
lamp stand, there is a minimum such lamp. Thus, there exists
$[\gamma]\in\partial_\infty L_2^+$ with infinite lamp stand equal to this
stabilization.

Take a subsequence $(g_{n_k})$ such that for every positive integer $K$, 
for all $k>K$ the lamps at positions at
most $K$ in the lamp stand for $g_{n_k}$ have achieved their eventual status and
$H(g_{n_k}) > K$.

Let $x\in L_2$. Set $K$ to be sufficiently large so that $K\geq \max\{m(x),
M(x), H(x)\}$, and for the finite values of $m(\gamma)$ and $m(\gamma, x)$,
$K\geq \max\{m(\gamma), m(\gamma,x)\}$ as well.

Let $k>K$. Assume that $h_{g_n}$ exists (and is therefore equal to $h_{g_{n_k}}$) and use Lemma \ref{distancelem} and Equation \ref{horodefeqn}:
\begin{align*}
h_{g_{n_k}}(x) =& \lim_{{n_k}\to\infty} 2(\max\{M(g_{n_k}, x)+1, H(g_{n_k}), H(x) \}\\&-\min\{m(g_{n_k}, x), H(x)\})-(H(g_{n_k})-H(x))\\ 
&- [2(\max\{M(g_{n_k})+1,
H(g_{n_k}),0 \}-\min\{m(\gamma),0\})-H(g_{n_k})]
\end{align*}
Notice that if $\max\{M(g_{n_k}), M(g_{n_k}, x)\} >
H(g_{n_k})$, then since $H(g_{n_k})>
M(x)$, we have that $M(g_{n_k}, x) = M(g_{n_k})$. 
Since $H(g_{n_k}) \geq \max\{H(x), 0\}$, we
have that $\max\{M(g_{n_k}, x)+1, H(g_{n_k}), H(x) \} = \max\{M(g_{n_k})+1,
H(g_{n_k}),0 \}$. Therefore,
$$h_{g_{n_k}}(x) = \lim_{{n_k}\to\infty} 2(\min\{m(\gamma),0\}-\min\{m(g_{n_k}, x), H(x)\})+H(x)$$
Now notice that if $m(g_{n_k})<0$ or $m(\gamma)<0$, then $m(\gamma) = m(g_{n_k})$. Similarly, if $m(g_{n_k}, x) < H(x)$ or $m(\gamma, x) < H(x)$, since $H(x)<K$, then $m(g_{n_k}, x) = m(\gamma, x)$.
So by Equation \ref{busemannposeqn} and the above, $h_{g_n} = \mathfrak{b}^\gamma$.\end{proof}

\begin{lemma}
\label{heightlem}

Suppose that a sequence $(g_n)\subset L_2$ is not left stable and $H(g_n) \to
+\infty$. If $(g_n)$ is associated with some horofunction $h_{g_n}$, then
$h_{g_n}=H$, the height function.

\end{lemma}

\begin{proof}

By Observation \ref{noboundsubseqobs}, $(g_n)$ has a
subsequence $(g_{n_i})$ such that $(m(g_{n_i}))$ is decreasing with
$m(g_{n_i})<-i$ for all $i$. We still have $H(g_{n_i}) \to
+\infty$, so we can further take a subsequence $(g_{n_k})$ such that for all
$k$, $m(g_{n_k})<-k$ and $H(g_{n_k})>k$.

Let $x\in L_2$, let $K=\max\{M(x),\ |m(x)|,\ |H(x)|\}$, and consider $k>K$.
By Lemma \ref{distancelem}, there exists $B \in \integers$ such that
$$d(g_{n_k}, x) = 2(B - m(g_{n_k}))-|H(g_{n_k}) - H(x)|$$
and 
$$d(g_{n_k}, id) = 2(B - m(g_{n_k}))-|H(g_{n_k})|.$$
Thus,
$$h_{g_{n_k}}(x) = \lim_{{n_k}\to\infty} |H(g_{n_k})|-|H(g_{n_k}) - H(x)| = H(x)$$ \end{proof}

\begin{lemma}
\label{busemannneglem}

Suppose that a sequence $(g_n)\subset L_2$ is right stable and $H(g_n) \to -\infty$. If $(g_n)$ is associated
with some horofunction $h_{g_n}$, then $h_{g_n}$ is equal to a Busemann function
$\mathfrak b^\gamma$ with $[\gamma]\in\partial_\infty L_2^-$.

\end{lemma}

\begin{proof}

As in the proof of Lemma \ref{busemannposlem}, but the Busemann function will have the lamplighter at $-\infty$ instead of $+\infty$.
\end{proof}

\begin{lemma}
\label{negheightlem}

Suppose that a sequence $(g_n)\subset L_2$ is not right stable and $H(g_n) \to -\infty$. If $(g_n)$ is associated with some horofunction $h_{g_n}$, then $h_{g_n}=-H$, the negation of the height function.

\end{lemma}

\begin{proof}

Similar to the proof of Lemma \ref{heightlem}.
\end{proof}

\begin{theorem}
\label{assumeconvergesthm}

Suppose that a sequence $(g_n)\subset L_2$ has $(H(g_n))$ converging to some
value $l\in\integers\cup\{\pm\infty\}$. If $(g_n)$ is associated with some horofunction $h_{g_n}$ and $g\in L_2$, then:\begin{enumerate}

	\item If $l\in\integers$ and $(g_n)$ is both left and right stable, 
then the sequence is eventually a constant value $g_0$ and 
$h_{g_n}$ is in the image of $L_2$ in $\overline{L_2}$.
$$h_{g_n}(g) = d(g, g_0)$$

	\item If $l = +\infty$ and $(g_n)$ is left stable, then $h_{g_n} = \mathfrak b^\gamma$ 
for some $[\gamma]\in\partial_\infty L_2^+$.
$$\mathfrak b^{+,\gamma}(g)=2(\min\{m(\gamma),0\}-\min\{m(\gamma,g),H(g)\})+H(g)$$

	\item If $l = -\infty$ and $(g_n)$ is right stable, then $h_{g_n} = \mathfrak b^\gamma$ for some $[\gamma]\in\partial_\infty L_2^-$.
$$\mathfrak b^{-,\gamma}(g)=2(\max\{M(\gamma,g)+1,H(g)\}-\max\{M(\gamma)+1,0\})-H(g)$$

	\item If $l\in\integers$ and $(g_n)$ is neither left nor right stable,
then $h_{g_n}=\mathfrak{s}^l$
$$\mathfrak{s}^l(g) = |l|-|l-H(g)|$$

	\item If $l\in\integers$ and $(g_n)$ is left--but not right--stable, 
	then $h_{g_n}=\mathfrak r^{+,f}$ for some $f\in L_2$.
$$\mathfrak r^{+,f}(g) = 2(\min\{m(f),l,0\} - \min\{m(f,g),H(g), l\}) +\mathfrak{s}^l(g)$$

	\item If $l\in\integers$ and $(g_n)$ is right--but not left--stable,
	then $h_{g_n}=\mathfrak r^{-,f}$ for some $f\in L_2$.
$$\mathfrak r^{-,f}(g) = 2(\max\{M(f,g)+1, H(g), l\} - \max\{M(f)+1, l, 0\}) +\mathfrak{s}^l(g)$$

	\item If $l = +\infty$ and $(g_n)$ is not left stable, then $h_{g_n} = H$.

	\item If $l = -\infty$ and $(g_n)$ is not right stable, then $h_{g_n} = -H$.
\end{enumerate}

\end{theorem}

\begin{proof}

If $l\in\integers$, then apply one of Lemmas \ref{bothboundslem},
\ref{oneboundlem}, or \ref{noboundslem}, as appropriate for the existence of
left or right stability. If $l = +\infty$, then
apply either Lemma \ref{busemannposlem} or \ref{heightlem}, depending on the
existence of left stability. If $l = -\infty$, then apply either Lemma
\ref{busemannneglem} or \ref{negheightlem}, depending on the existence of 
right stability.
\end{proof}

\begin{lemma}
\label{lamplighterconvergeslem}

Suppose for a sequence $(g_n)\subset L_2$, $(H(g_n))$ does not converge in $\mathbb{Z}\cup \{\pm\infty\}$. Then $(g_n)$ is \emph{not} associated with a horofunction.

\end{lemma}

\begin{proof}
By our hypotheses, $(g_n)$ has subsequences $(g_{n_i})$ and $(g_{n_j})$ such
that $(H(g_{n_i}))$ and $(H(g_{n_j}))$ converge in
$\integers\cup\{\pm\infty\}$, but to distinct values.  By Theorem
\ref{assumeconvergesthm} and Observation \ref{nonbusemannobs}, since 
these limits are distinct, $h_{g_{n_i}} \neq h_{g_{n_j}}$.
Thus $h_{g_n}$ cannot exist. 
\end{proof}

\begin{corollary}
\label{bigcor}
Let $h\in \overline{L_2}$, and choose a sequence $(g_n)\subset L_2$ such that $h
= h_{g_n}$. Then $(H(g_n))$ converges to some value $l\in\integers\cup\{\pm\infty\}$, and $h_{g_n}$ can be categorized as in Theorem \ref{assumeconvergesthm}
\end{corollary}

\section{Topology of the horofunction boundary}
\label{topologysec}

The topology of $\partial_h L_2$ is the topology of uniform convergence on compact sets. 
The standard basis is the collection of sets of the form
$$B_K(h, \epsilon)=\{h'\in \partial_h L_2\ | \ |h(x) - h'(x)|<\epsilon \text{ for all } x\in K\}$$
where $K\subset L_2$ is compact and $\epsilon>0$.
By restricting to $0<\epsilon< 1$, we obtain an equivalent basis.
Since the minimum distance between distinct points in $L_2$ is 1, we may use the following sets as a basis:
$$B_K(h) = \{h'\in \partial_h L_2\ | \ h(x) = h'(x)\text{ for all }x\in K\}$$
where $K\subset L_2$ is finite. 
Notice that pointwise convergence implies convergence in our topology 
since compact sets of $L_2$ are finite.

With the explicit descriptions of the horofunctions found in Section
\ref{modelssec}, we can establish the accumulation points of $\partial_h L_2$.
We begin by recalling that since $\mathfrak s^l(g) = |l| - |l- H(g)|$, we have
\begin{observation}
\label{Haccumofspineobs}
$\mathfrak s^l \rightarrow \pm H$ as $l\rightarrow\pm\infty$.
\end{observation}

\begin{observation}
\label{subspacesembedobs}
The injective map that takes elements of $\partial_\infty L_2^+$ to their
Busemann functions in $\partial_h L_2$ is continuous, and the same is true of
$\partial_\infty L_2^-$.
\end{observation}

Contrast this result with Observation \ref{notcontinuousobs}, which states that the injection of the union of these two sets into the horofunction boundary is not continuous.
Recall that the obstruction to continuity was the non-Hausdorff property,
which was proved by finding neighborhoods of
distinct elements of $\partial_\infty L_2^+$ that always shared elements of
$\partial_\infty L_2^-$.

\begin{proof}

Let $[\gamma]\in \partial_\infty L_2^+$, and consider $B_K(\mathfrak b^\gamma)$ for
some finite $K\subset L_2$.  Let $M=\max\{M(g), H(g)\ | \ g\in K\}$, and let
$k\in \Z$ such that $k>M+2|m(\gamma)|$ if $m(\gamma)<0$ or $k>M$ otherwise.
Consider the set $$B_{[0,k]}([\gamma], \epsilon) = \{\gamma'\in \partial_\infty
L_2^+\ |\ \sup\{d(\gamma(x), \gamma'(x)) \ | \ x\in [0,k]\}<\epsilon\}$$ for
$0<\epsilon<1$.  In \cite[Observation 4.1]{joneskelsey}, the authors observe that
$B_{[0,k]}([\gamma], \epsilon)$ is an open set in $\partial_\infty L_2^+$.
Notice that if $\gamma'\in B_{[0,k]}([\gamma], \epsilon)$, then the lamp stands
of $\gamma$ and $\gamma'$ agree on all lamps at positions $M$ or below.  Thus,
by Equation \ref{busemannposeqn}, $\mathfrak b^\gamma(g) = \mathfrak b^{\gamma'}(g)$ for all
$g\in K$.  Therefore, $\mathfrak b^{\gamma'} \in B_K(\mathfrak b^\gamma)$ for all $\gamma'\in
B_{[0,k]}([\gamma], \epsilon)$, and so our injection is continuous.  

The proof for the injection of $\partial_\infty L_2^-$ is similar.
\end{proof}

The topology of each of these sets is a punctured Cantor set, but in
$\partial_h{L_2}$ these punctures are ``filled'' by the
height function and its negative, as we now show.

\begin{observation}
\label{Haccumofbusemannobs}
If $([\gamma_n]) \subset \partial_\infty L_2^+$ with $m(\gamma_n) \to -\infty$, then
$\mathfrak b^{\gamma_n} \to H$. Similarly, if $([\gamma_n])\subset \partial_\infty
L_2^-$ with $M(\gamma_n) \to +\infty$, then $\mathfrak b^{\gamma_n} \to -H$.

\end{observation}

\begin{proof}
Let $([\gamma_n])\subset \partial_\infty L_2^+$ with $\lim m(\gamma_n) = -\infty$. By
Equation \ref{busemannposeqn}, 
$$\mathfrak b^{\gamma_n}(g) = 
2\left(\min\{m(\gamma_n),0\} - \min\{m(\gamma_n,g),H(g)\}\right)+H(g)$$ 
Fix $g$
and take $n$ large enough so that $m(\gamma_n) < \min\{0, m(g), H(g)\}$, then 
$$\mathfrak b^{\gamma_n}(g) = 
2\left(m(\gamma_n) - m(\gamma_n)\right)+H(g)  = H(g)$$ 
Thus, $\mathfrak b^{\gamma_n}\to H$. The other proof is similar.
\end{proof}

For a given $l \in \integers$, the following observation remarks that the
spine is an accumulation point of the positive and negative rib
functions. The proofs are calculations similar to those in Observation
\ref{Haccumofbusemannobs}.

\begin{observation}
\label{spineaccumofribsobs}
Let $(f_n^l) \subset L_2$ be a sequence satisfying $M(f_n^l) < H(f_n^l)=l$ and
$m(f_n^l) \rightarrow -\infty$ as $n\rightarrow \infty$. 
Then $\mathfrak r^{+,f_n^l} \rightarrow \mathfrak s^l$. 

Similarly, if $(f_n^l) \subset L_2$ is a sequence satisfying
$m(f_n^l) \geq H(f_n^l)=l$ and $M(f_n^l) \rightarrow \infty$ as $n\rightarrow \infty$,
then $\mathfrak r^{-,f_n^l} \rightarrow \mathfrak s^l$.
\end{observation}

Finally, the ribs accumulate to Busemann functions:
\begin{observation}
\label{busemannaccumofribsobs}
For a geodesic ray $\gamma$, with $\gamma(0) = id$, set $f_n=\gamma(n)$.
If $\gamma \in \partial_\infty L_2^+$, then for large enough $n$, each $f_n$ 
defines $\mathfrak r^{+,f_n}$ and
$\mathfrak r^{+,f_n} \rightarrow\mathfrak b^{+,\gamma}$. 
If $\gamma \in \partial_\infty L_2^-$, then for large enough $n$ each $f_n$ defines
$\mathfrak r^{-,f_n}$ and $\mathfrak r^{+,f_n} \rightarrow \mathfrak
b^{-,\gamma}$.  
\end{observation}

\begin{proof}
We consider the $\gamma \in \partial_\infty L_2^+$ case.
For large enough $n$, each $f_n$ satisfies the requirements for defining 
$\mathfrak r^{+,f_n}$.  Let $g \in L_2$ be given, and
consider Equations \ref{posribeqn} and \ref{busemannposeqn}.
Again for large enough $n$, $m(f_n) = m(\gamma)$ and $m(f_n,g) = m(\gamma,g)$. 
Thus
\[ \mathfrak r^{+,f_n} - \mathfrak b^{+,\gamma}  = \mathfrak s^{H(f_n)}(g) - H(g) \rightarrow 0
\textrm{ as } n\rightarrow \infty.\] 
\end{proof}

With Observations \ref{Haccumofspineobs}, \ref{Haccumofbusemannobs},
\ref{spineaccumofribsobs}, and \ref{busemannaccumofribsobs}, we have the picture
of the horofunction boundary illustrated in Figure \ref{boundarypic} in the introduction.

\section{Action of $L_2$ on the horofunction boundary}

\label{actionsec}
We now conclude with a few comments about the action of $L_2$ on $\partial_h L_2$.

An isometric action of a group $G$ on a metric space $(X,d)$ with base point $b$ can be extended to the horofunction boundary $\hbdry X$ in the following way: For $g\in G$ and $(y_n)\subset X$ giving rise to a horofunction, we have that
$$g\cdot h_{y_n}(x) = h_{g\cdot y_n}(x) = \lim_{n\to\infty}d(g\cdot y_n, x)-d(g\cdot y_n, b)$$

In our setting, the action of $L_2$ on itself is by left multiplication.
We compose lamp stands $g_1\cdot g_2$ by starting with the lamp stand for $g_1$
and having the lamplighter move and toggle lamps as in $g_2$, but 
from a starting position of $H(g_1)$ rather than $0$.

\begin{observation}
\label{actionobs}

Let $g\in L_2$, $h\in \overline{L_2}$, and choose $(g_n)\subset L_2$ such that $h=h_{g_n}$.
Then $H(g\cdot g_n) \rightarrow H(g) + \lim H(g_n)$, where for $k\in\integers$,
$\pm \infty + k$ is understood to mean $\pm \infty$. Also, $(g \cdot g_n)$ is
left (resp. right) stable, iff $(g_n)$ is left (resp. right) stable.
\end{observation}

\begin{proof}
These statements all follow from the fact that the lamp stand for $g$ has only finitely many lit lamps and the lamplighter at a finite position.
\end{proof}

\begin{corollary}
\label{actioncor}

Each of the categories of horofunctions in $\overline{L_2}$ 
described in Theorem \ref{assumeconvergesthm}
is invariant under the action of $L_2$. 

\end{corollary}

\begin{proof}

This result follows from Observation \ref{actionobs} and Corollary \ref{bigcor}.
\end{proof}

Interestingly, this implies the following:

\begin{corollary}
\label{fixedeheightcor}

The height function $H$ and its negation are global fixed points of the action of $L_2$ on $\hbdry L_2$.

\end{corollary}

We now consider the action of $L_2$ on each of the other categories of horofunctions. 

Let $g\in L_2$.
The action of $g$ on $\partial_\infty L_2$ is described in \cite[\S 3.4 and \S 4.6]{joneskelsey}. 
If $H(g)\neq 0$, then the action of $g$ on $\partial_\infty L_2$ has two fixed points, 
which are given the notation $g^\infty$ and $g^{-\infty}$ in \cite{joneskelsey}.
If $H(g)>0$, then $g^\infty\in\partial_\infty L_2^+$ and $g^{-\infty}\in\partial_\infty L_2^-$. 
Otherwise, the reverse is true.
In the topology of $\partial_\infty L_2$, the action of $g$ has north-south dynamics
with attractor $g^\infty$ and repeller $g^{-\infty}$.
Recall that in $\partial_\infty L_2$, the punctures in the two Cantor sets are ``filled''
by points from the opposite Cantor set, while in $\partial_h L_2$, these punctures 
are filled by $H$ and $-H$.
Thus, in the horofunction boundary we see similar dynamics with the visual boundary,
except it occurs on the separate sets of
$\partial_\infty L_2^+ \cup \{H\}$ and $\partial_\infty L_2^- \cup \{-H\}$.

\begin{observation}

For $g\in L_2$ with $H(g) \neq 0$, the action of $g$ on $\hbdry L_2$ has four
fixed points: $H, -H, \mathfrak b^{g^\infty}, \mathfrak b^{g^{-\infty}}$.
The action of $g$  has north-south dynamics on $\partial_\infty L_2^+\cup \{H\}$ 
with poles $H$ and either $g^\infty$ or $g^{-\infty}$ (whichever is in the set)
and also on $\partial_\infty L_2^-\cup \{-H\}$
with poles $-H$ and either $g^\infty$ or $g^{-\infty}$ (whichever is in the set).
The point $g^\infty$ is always an attractor and the point $g^{-\infty}$ is always a repeller.
If $H(g) >0$, then $H$ is an attractor and $-H$ is a repeller.
If $H(g)<0$, then these roles are reversed.

\end{observation}

For a spinal horofunction $\mathfrak s^l \in \hbdry L_2$, $l\in \integers$, the
action of $g$ on $\mathfrak s^l$ is given by 
$g\cdot \mathfrak s^l = \mathfrak s^{H(g) + l}$.

We see similar behavior on the ribs of $\hbdry L_2$ in that the $l$ value is translated
by the height of the group element, but there is also
additional structure in this case:
Let $g\in L_2$ and let $f\in L_2$ such that $\mathfrak{r}^{+,f}$ exists (i.e. $M(f)<H(f)$).
Then $g\cdot \mathfrak{r}^{+,f} = \mathfrak{r}^{+,\overline{gf}}$ where
$\overline{gf}$ has the lamp stand for $gf$ but with all of the lamps at position
$H(gf) = H(g)+H(f)$ and above switched off.
Note that $g$ acts as a bijection from $R_{H(f)}^+$ to $R_{H(g)+H(f)}^+$.

Notice that if $m(g) \neq H(g)+m(f)$, then
$$m(\overline{gf}) = \min\{H(g)+m(f), m(g)\}$$
Using the notation in Section \ref{ribssubsec}, the above yields
the following description of the action on rib horofunctions that are ``close'' to the spine.

\begin{observation}
\label{invblobsobs}
Let $g\in  L_2$ and $l\in\integers$.
Let $k<l$ such that $H(g)+k<m(g)$.
The action of $g$ on $\partial_h L_2$ restricted to the subset $\hat{m}_l^{-1}(k)$ of $R_l^+$
is a bijection onto the subset $\hat{m}_{H(g)+l}^{-1}(H(g)+k)$ of $R_{H(g)+l}^+$
\end{observation}

\begin{corollary}
\label{invspinenbhdcor}
Let $g\in L_2$ such that $H(g)= 0$ and let $l\in \integers$.
If $k < \min\{m(g)-H(g),l\}$, then the subset $\hat{m}_l^{-1}(k)$ of $R_l^+$ 
is invariant under the action of $g$.
\end{corollary}

The action on such a rib $\mathfrak{r}^{+,f}$ leaves $m(f)$ and $H(f)$ fixed, 
but changes the status of lamps between those positions.
This gives a permutation on the set $m_l^{-1}(k)$.

The similar statements also hold for negative ribs.

\bibliographystyle{plain}
\bibliography{joneskelseyhoroboundarydl}
\end{document}